\documentclass[12pt]{amsart}

\usepackage{amsmath}
\usepackage{amssymb}
\usepackage{amsthm}
\usepackage{amscd}

\newtheorem{thm}{Theorem}[section]
\newtheorem{lem}[thm]{Lemma}
\newtheorem{prop}[thm]{Proposition}
\newtheorem{cor}[thm]{Corollary}
\newtheorem*{thmA}{Theorem A}
\newtheorem*{propB}{Proposition B}

\newtheorem*{GGC}{Greenberg's generalized conjecture}

\theoremstyle{definition}
\newtheorem{definition}[thm]{Definition}
\newtheorem{example}[thm]{Example}

\theoremstyle{remark}
\newtheorem{rem}[thm]{Remark}

\newcommand\Gal{\mathrm{Gal}}
\newcommand\zprank{\mathop{\mathrm{rank}_{\mathbb{Z}_p}}}

\begin{document}

\title{A remark on Greenberg's generalized conjecture for imaginary 
$S_3$-extensions of $\mathbb{Q}$}
\author{Tsuyoshi Itoh}

\begin{abstract}
Let $K/ \mathbb{Q}$ be an imaginary $S_3$-extension, 
and $p$ a prime number which splits into exactly three primes in $K$.
We give a sufficient condition 
for the validity of Greenberg's generalized conjecture for $K$ and $p$.
\end{abstract}

\maketitle

\section{Introduction}\label{Introduction}

Let $p$ be a prime number.
We recall the statement of Greenberg's generalized conjecture (GGC) 
for an algebraic number field $K_\circ$ and $p$.
We denote by $\widetilde{K_\circ}$ the composite of all 
$\mathbb{Z}_p$-extensions of $K_\circ$.
Then, the Galois group $\Gal (\widetilde{K_\circ}/ K_\circ)$ 
is topologically isomorphic to $\mathbb{Z}_p^{\oplus d}$ with some positive integer $d$.
Let $L (\widetilde{K_\circ})/ \widetilde{K_\circ}$ be the 
maximal unramified abelian pro-$p$ extension.
We put $X (\widetilde{K_\circ}) = 
\Gal (L (\widetilde{K_\circ})/ \widetilde{K_\circ})$. 
We denote by $\Lambda_{\Gal (\widetilde{K_\circ}/ K_\circ)}$ the 
completed group ring $\mathbb{Z}_p [[ \Gal (\widetilde{K_\circ}/ K_\circ)]]$.
It is well known that $X (\widetilde{K_\circ})$ is a 
finitely generated torsion module over 
$\Lambda_{\Gal (\widetilde{K_\circ}/ K_\circ)}$ (see \cite[Theorem 1]{Gre73}).

\begin{GGC}[{\cite[Conjecture (3.5)]{Gre01}}]
$X (\widetilde{K_\circ})$ is a pseudo-null 
$\Lambda_{\Gal (\widetilde{K_\circ}/K_\circ)}$-module. 
That is, there are two relatively prime elements of 
$\Lambda_{\Gal (\widetilde{K_\circ}/K_\circ)}$ such that 
both annihilate $X (\widetilde{K_\circ})$.
\end{GGC}

Minardi \cite{Min} studied GGC, mainly for imaginary quadratic fields.
After that, many authors gave sufficient conditions for the validity of GGC 
in various situations.
For example, see the author \cite{I2011}, Fujii \cite{Fujii}, Kleine \cite{Kle}, 
Kataoka \cite{Kata}, Takahashi \cite{Taka}.
In \cite[Remark 1.3]{Taka}, several known results are stated in detail.
See also Assim-Boughadi \cite{A-B} and the results referred there.
We also mention Murakami \cite{Mura} as a recent result.

In the present paper, we consider GGC for the following $K$ and $p$.
Let $K/\mathbb{Q}$ be an imaginary  
Galois extension whose Galois group is isomorphic to the 
symmetric group $S_3$ of degree $3$.
We assume that $p$ splits into three primes $\mathfrak{P}_1$, $\mathfrak{P}_2$, $\mathfrak{P}_3$ in $K/\mathbb{Q}$.

There is a unique imaginary quadratic field $k$ contained in $K$.
In our situation, $p$ is not decomposed in $k$, and the unique prime of 
$k$ lying above $p$ is completely decomposed in $K$.
For each $i \in \{ 1,2,3 \}$, let 
$K_{\mathfrak{P_i}}$ be the completion of $K$ 
at $\mathfrak{P}_i$.
Then $[ K_{\mathfrak{P}_i} : \mathbb{Q}_p ] =2$ for every $i$.

The decomposition field $F$ of $K/\mathbb{Q}$ for 
$\mathfrak{P}_1$ is a complex cubic field.
We denote by $\mathfrak{p}$ the prime of $F$ lying below $\mathfrak{P}_1$.
There is just one more prime $\mathfrak{p}^*$ of $F$ lying above $p$, which splits 
into two primes $\mathfrak{P}_2$, $\mathfrak{P}_3$ in $K/F$.
We note that $[ F_{\mathfrak{p}} : \mathbb{Q}_p ] =1$ 
and $[ F_{\mathfrak{p}^*} : \mathbb{Q}_p ] =2$, 
where $F_{\mathfrak{p}}$  
(resp.~$F_{\mathfrak{p}^*}$) is the completion of $F$ 
at $\mathfrak{p}$ (resp.~$\mathfrak{p}^*$).
It is known that there exists a unique $\mathbb{Z}_p$-extension
$N^* /F$ unramified outside $\mathfrak{p}^*$ (see \cite[Lemma 3.4 (2)]{Kata}).

Our main result is the following:

\begin{thm}\label{main_thm}
Let the notation be as above.
Moreover, 
let $M_{\mathfrak{P}_3} (K)/K$ be the maximal abelian pro-$p$ extension 
unramified outside $\mathfrak{P}_3$.
Assume that all of the following conditions are satisfied:\\
(C1) $\mathfrak{p}$ is finitely decomposed in $N^*$, \\
(C2) $\mathfrak{p}^*$ is not decomposed in $N^*$, and \\
(C3) $\mathfrak{P}_2$ is not decomposed in $M_{\mathfrak{P}_3} (K)$.\\
Then, GGC for $K$ and $p$ holds.
\end{thm}

\begin{rem}\label{rem_C1}
$M_{\mathfrak{P}_3} (K)/K$ is known to be a finite extension 
(see Section \ref{Preliminaries}).
The condition (C1) is equivalent to \cite[Assumption 3.1]{Kata} 
(when the complex cubic field has two primes dividing $p$).
It seems unknown whether (C1) always holds or not 
(see also \cite[Remark 3.2 (2)]{Kata}, \cite[Sections 7.2 and 7.3]{Hachi}).
However, (C1) was confirmed to be satisfied for many cases 
(see Remark \ref{rem_C1_2}).
\end{rem}

\begin{rem} 
Kataoka \cite{Kata} gave sufficient conditions for the validity of 
GGC for complex cubic fields 
(he treated all cases of decomposition on $p$).
Compare our Theorem \ref{main_thm} with \cite[Theorem 3.3 (1), (2)]{Kata}.
In particular, by using \cite[Theorem 3.3 (2)]{Kata}, 
we see that GGC for (our) $F$ and $p$ also holds 
under the assumptions (C1), (C2), (C3). 
To show \cite[Theorem 3.3 (2)]{Kata}, 
Kataoka used $N^*/F$ as the first step of his proof 
(\cite[Proposition 3.5 (2)]{Kata}).  
We use $N^* K/K$ as the first step. 
However, $\widetilde{F}$ is not used in our proof.
\end{rem}

\begin{rem}\label{rem_contained}
Assume that $p=3$.
When $K$ is contained in $\widetilde{k}$, 
the validity of GGC for $k$ and $p$ implies the validity of GGC for $K$ and $p$. 
See \cite[p.33, Remarks (ii)]{Min}.
However, in our case, if $K$ is contained in $\widetilde{k}$, then 
$K/k$ is an unramified extension, hence the validity of GGC for $k$ and $p$ seems non-trivial
(see, e.g., \cite[Section 3.D]{Min}).
Incidentally, for the case where $p=2$, we see that 
$K$ is never contained in $\widetilde{F}$ 
because the real archimedean prime of $F$ ramifies in $K$.
For a somewhat related result, see also \cite[Corollary 4.4]{Kle}.
\end{rem}

We give several preparations in Sections \ref{Preliminaries}.
We shall show Theorem \ref{main_thm} in Section \ref{Proof_main_thm}.
We will give examples in Section \ref{Examples}.

\section{Preliminaries}\label{Preliminaries}
First, we define several notation and recall well known facts.
We denote by $|A|$ the cardinality of a set $A$.
We also denote by $\zprank B$ the $\mathbb{Z}_p$-rank of a finitely generated 
$\mathbb{Z}_p$-module $B$.
For a pro-$p$ group $G$ which is topologically isomorphic to $\mathbb{Z}_p^{\oplus d}$ with some 
positive integer $d$, let $\Lambda_G$ be the the 
completed group ring $\mathbb{Z}_p [[ G ]]$.
We use the notation defined in Section \ref{Introduction}.

In this paragraph, we denote by $\mathcal{K}$ an algebraic extension of $K$.
Let $L (\mathcal{K})/ \mathcal{K}$ be the maximal unramified abelian pro-$p$ extension, 
and put $X (\mathcal{K}) = \Gal (L (\mathcal{K})/ \mathcal{K})$.
We denote by $S_p$ the set $\{ \mathfrak{P}_1, \mathfrak{P}_2, \mathfrak{P}_3\}$.
For a non-empty subset $S$ of $S_p$, let $M_S (\mathcal{K})/ \mathcal{K}$ be 
the maximal abelian pro-$p$ extension unramified outside $S$, and put 
$\mathfrak{X}_S (\mathcal{K}) = \Gal (M_S (\mathcal{K})/ \mathbb{K})$.
For the case where $S = \{ \mathfrak{P}_i \}$,
we shall write them $M_{\mathfrak{P}_i} (\mathcal{K})$, 
$\mathfrak{X}_{\mathfrak{P}_i} (\mathcal{K})$ in short.

For $i \in \{ 1,2,3 \}$, let $\mathcal{U}_i$ be the group of 
principal units in $K_{\mathfrak{P}_i}$, 
and put $\mathcal{U} = \bigoplus_{i=1}^3 \mathcal{U}_i$.
Let $E_K$ be the group of units in $K$, 
and put 
\[ E'_K = \{ \varepsilon \in E_K \; | \; 
\text{$\varepsilon \equiv 1 \pmod{\mathfrak{P}_i}$ for every $i \in \{ 1, 2, 3 \}$} \}. \]
We denote by $\mathcal{E}$ the image of the mapping 
$E'_K \otimes_{\mathbb{Z}} \mathbb{Z}_p 
\to \mathcal{U}$ induced from the diagonal embedding
\[ E'_K \to \mathcal{U}= \mathcal{U}_1 \oplus \mathcal{U}_2 \oplus \mathcal{U}_3, 
\quad \varepsilon \mapsto
(\varepsilon, \varepsilon, \varepsilon) \]
(see, e.g., \cite[Appendix]{IwaU4}).
By class field theory, $\mathcal{U}/ \mathcal{E}$ is isomorphic to 
$\Gal (M_{S_p} (K)/L(K))$.
It is known that Leopoldt's conjecture holds for $K$ and $p$ 
since $K$ is an abelian extension of an imaginary quadratic field 
(see \cite{Bru}).
Then, $\zprank \mathcal{E}$ is $2$, which is equal to the free rank of $E_K$.
Since $\zprank \mathcal{U}$ is $6$, we see that 
$\zprank \Gal (M_{S_p} (K)/L(K))$ is $4$. 
Hence, $\widetilde{K}/K$ is a $\mathbb{Z}_p^{\oplus 4}$-extension.

Next, we shall state several results which will be used in the proof of Theorem \ref{main_thm}.
Recall that $N^*/F$ is the $\mathbb{Z}_p$-extension unramified 
outside $\mathfrak{p}^*$.
We put $N^{(1)} = N^* K$, which is a $\mathbb{Z}_p$-extension of $K$.
The following result is crucial to prove Theorem \ref{main_thm}.

\begin{thmA}[Maire \cite{Mai}]
For a non-negative integer $n$, let $N^{(1)}_n$ be the 
$n$th layer of $N^{(1)}/K$.
Then for every $n$, 
$\mathfrak{X}_{\mathfrak{P}_1} (N^{(1)}_n)$ is finite.
In particular, $\mathfrak{X}_{\mathfrak{P}_1} (K)$ is finite.
\end{thmA}

\begin{proof}
We recall the following facts:
$N^{(1)}_n / F$ is an abelian extension, 
$\mathfrak{P}_1$ is the unique prime of $K$ lying above $\mathfrak{p}$,  
$[ F_{\mathfrak{p}} : \mathbb{Q}_p ]=1$, and $[F : \mathbb{Q}] =3$.
Then we can apply \cite[Theorem 25]{Mai}, and the assertion follows.
\end{proof}

We can also see that both $\mathfrak{X}_{\mathfrak{P}_2} (K)$ and 
$\mathfrak{X}_{\mathfrak{P}_3} (K)$ are finite.

We note that Hachimori \cite{Hachi} treated similar situation to ours.
In particular, it seems that he essentially showed a result similar to Theorem A 
for the cyclotomic $\mathbb{Z}_p$-extension of $K$
(see the proofs of \cite[Theorems 6.2 and 7.3]{Hachi}). 

\begin{cor}\label{rank_Zp2_ext}
We put $S = \{ \mathfrak{P}_i, \mathfrak{P}_j \}$ 
with $1 \leq i < j \leq 3$.
Then $\zprank \mathfrak{X}_S (K)$ is $2$.
\end{cor}

\begin{proof}
Since $K/\mathbb{Q}$ is a Galois extension, it is sufficient to 
show only for the case where  
$S = \{ \mathfrak{P}_1, \mathfrak{P}_2 \}$.
This corollary follows from the basic 
result on abelian pro-$p$ extensions with restricted ramification 
(see, e.g., \cite[Theorem 5]{Mai}, \cite[Proposition 3.1]{Hachi}).
Theorem A (for $K$) asserts that 
the image of $E'_K \otimes_{\mathbb{Z}} \mathbb{Z}_p 
\to \mathcal{U}_1$ has $\mathbb{Z}_p$-rank $2$.
Then the image of the mapping 
\[ E'_K \otimes_{\mathbb{Z}} \mathbb{Z}_p 
\to \mathcal{U}_1 \oplus \mathcal{U}_2 \]
induced from the diagonal embedding also has $\mathbb{Z}_p$-rank $2$.
Since $\zprank (\mathcal{U}_1 \oplus \mathcal{U}_2)$ is $4$, 
the assertion has been shown.
\end{proof}

\begin{definition}
Let $\mathcal{K}/K$ be an abelian (finite or infinite) extension,
and $\mathcal{K}'$ an intermediate field of $\mathcal{K}/K$.
For $i \in \{1,2,3 \}$, we define the following symbols.
\begin{itemize}
\item $D_i (\mathcal{K}/\mathcal{K}')$ : the decomposition subgroup 
of $\Gal (\mathcal{K}/\mathcal{K}')$ for a prime lying above $\mathfrak{P}_i$.
\item $I_i (\mathcal{K}/\mathcal{K}')$ : the inertia subgroup 
of $\Gal (\mathcal{K}/\mathcal{K}')$ for a prime lying above $\mathfrak{P}_i$.
\end{itemize}
Since $\mathcal{K}/K$ is abelian, these groups are uniquely determined 
independent on the choice of a prime lying above $\mathfrak{P}_i$.
\end{definition}

\begin{lem}\label{decomp_inertia_rank}
Assume that the condition (C1) in Theorem \ref{main_thm} is satisfied.
Then, 
\[ \zprank I_i (\widetilde{K}/K) = 2 \quad \text{and} \quad 
\zprank D_i (\widetilde{K}/K) = 3 \]
for every $i \in \{1,2,3 \}$.
\end{lem}

\begin{proof}
Since $\zprank \mathcal{U}_i$ is $2$, we see that 
$\zprank I_i (\widetilde{K}/K)$ is at most $2$, 
and $\zprank D_i (\widetilde{K}/K)$ is at most $3$.
We shall construct a $\mathbb{Z}_p^{\oplus 3}$-extension 
such that the inertia and decomposition subgroups have the maximal 
$\mathbb{Z}_p$-rank.

It is sufficient to show the assertion for $i=1$.
By Corollary \ref{rank_Zp2_ext}, there is a unique 
$\mathbb{Z}_p^{\oplus 2}$-extension $N^{\sharp} / K$ 
unramified outside $\{ \mathfrak{P}_1, \mathfrak{P}_2 \}$.
By Theorem A, we see that 
$\zprank I_1 (N^{\sharp} / K)$ must be $2$.
From this fact, we also see that 
$N^{(1)} \cap N^{\sharp} /K$ is a finite extension
(note that $N^{(1)}/K$ is unramified at $\mathfrak{P}_1$).
Then, $N^{(1)} N^{\sharp} /K$ is a $\mathbb{Z}_p^{\oplus 3}$-extension.
We assumed that (C1) is satisfied, that is, 
$\mathfrak{P}_1$ is finitely decomposed in $N^{(1)} /K$.
Combining these facts, we see that 
\[ \zprank I_i (N^{(1)} N^{\sharp}/K) = 2 \quad \text{and} \quad 
\zprank D_i (N^{(1)} N^{\sharp}/K) = 3. \]
The assertion of this lemma follows from the facts stated 
in the previous paragraph.
\end{proof}

\begin{rem}
We do not need the satisfaction of (C1) to show $\zprank I_i (\widetilde{K}/K) = 2$.
(There is another approach to show this fact by using $\widetilde{k}/k$. 
See \cite[Section 3]{Min}.)
One can show that the infiniteness of $D_1 (\widetilde{K}/K) / I_1 (\widetilde{K}/K)$ 
is equivalent to the satisfaction of (C1). 
Note that Minardi also considered the infiniteness of
$D_i (\widetilde{K}/K) / I_i (\widetilde{K}/K)$ for certain $S_3$-extensions $K/\mathbb{Q}$
(see \cite[Sections 3, 6]{Min}). 
In it, he mentioned the equivalence of the infiniteness of 
$D_i (\widetilde{K}/K) / I_i (\widetilde{K}/K)$ 
and the validity of the conjecture called ``Jaulent's conjecture'' 
(a conjecture which seems to be derived from \cite[p.155, Conjecture]{Jau}) there.
\end{rem}

\begin{prop}\label{key_prop}
(i) Assume that the condition (C1) in Theorem \ref{main_thm} is satisfied.
Then, $\zprank (D_2 (\widetilde{K}/K) \cap D_3 (\widetilde{K}/K))$ is $2$. \\
(ii) $I_1 (\widetilde{K}/K) \cap D_2 (\widetilde{K}/K) \cap D_3 (\widetilde{K}/K)$ 
is trivial.
\end{prop}

\begin{proof}
We shall show (i). 
Let $K^{\mathfrak{P}_2}$ (resp.~$K^{\mathfrak{P}_3}$) be the decomposition 
field of $\widetilde{K}/K$ for $\mathfrak{P}_2$ (resp.~$\mathfrak{P}_3$).
As a consequence of Lemma \ref{decomp_inertia_rank},
we see that both $\zprank \Gal (K^{\mathfrak{P}_2}/K)$ and 
$\zprank \Gal (K^{\mathfrak{P}_3}/K)$ are $1$.
By using Theorem A, we can see that 
$K^{\mathfrak{P}_2} \cap K^{\mathfrak{P}_3}/K$ is a finite extension.
Hence, the $\mathbb{Z}_p$-rank of $\Gal (K^{\mathfrak{P}_2} K^{\mathfrak{P}_3}/K)$ 
is $2$.
Since $K^{\mathfrak{P}_2} K^{\mathfrak{P}_3}$ corresponds to 
$D_2 (\widetilde{K}/K) \cap D_3 (\widetilde{K}/K)$, we obtain (i).

To show (ii), we first give several preparations.
Let $\sigma, \tau_1$ be an elements of $\Gal (K/\mathbb{Q})$ such that 
$\sigma$ generates $\Gal (K/k)$ and $\tau_1$ generates $\Gal (K/F)$.
Assume that $\sigma (\mathfrak{P}_1) = \mathfrak{P}_2$ and 
$\sigma^2 (\mathfrak{P}_1) = \mathfrak{P}_3$.
Then $\sigma \tau_1 \sigma^{-1}$ fixes $\mathfrak{P}_2$ and 
$\sigma^2 \tau_1 \sigma^{-2}$ fixes $\mathfrak{P}_3$.

We can take $\varepsilon_1 \in E'_K$ satisfying the following conditions:
$\tau_1 (\varepsilon_1) =\varepsilon_1$,  
$\varepsilon_1 \sigma (\varepsilon_1) \sigma^2 (\varepsilon_1)=1$, 
and $\varepsilon_1, \sigma(\varepsilon_1)$ are multiplicative independent.
We put $\varepsilon_2 = \sigma (\varepsilon_1)$ and 
$\varepsilon_3 = \sigma^2 (\varepsilon_1)$.
Since Leopoldt's conjecture holds for $K$ and $p$, 
we see that the image of 
\[ \langle \varepsilon_1, \varepsilon_2 \rangle \otimes_{\mathbb{Z}} 
\mathbb{Z}_p \to \mathcal{U} \]
has $\mathbb{Z}_p$-rank $2$.

We can also take a $\mathfrak{P}_1$-unit $\pi_1$ of $K$ 
satisfying the following conditions: 
$\pi_1$ generates a positive power of $\mathfrak{P}_1$,
$\tau_1 (\pi_1) =\pi_1$, and $\pi_1 \equiv 1 \pmod{\mathfrak{P}_i}$ for every $i \in \{2,3 \}$.
We put $\pi_2 = \sigma (\pi_1)$ and 
$\pi_3 = \sigma^2 (\pi_1)$.
Then $\pi_2$ (resp.~$\pi_3$) is a $\mathfrak{P}_2$-unit 
(resp.~$\mathfrak{P}_3$-unit).
We shall define the subgroups $\mathcal{D}_2$, $\mathcal{D}_3$ of 
$\mathcal{U} = \mathcal{U}_1 \oplus \mathcal{U}_2 \oplus \mathcal{U}_3$ 
as the following:
\[ \mathcal{D}_2 = \{ (\pi_2^x, u_2, \pi_2^x) \; | \; 
x \in \mathbb{Z}_p, u_2 \in \mathcal{U}_2 \}, \quad
\mathcal{D}_3 = \{ (\pi_3^y, \pi_3^y, u_3) \; | \; 
y \in \mathbb{Z}_p, u_3 \in \mathcal{U}_3 \}. \]
By class field theory, the image of 
$\mathcal{D}_2 \to \mathcal{U}/ \mathcal{E}$ 
(resp.~$\mathcal{D}_3 \to \mathcal{U}/ \mathcal{E}$)
corresponds to a finite index subgroup of 
$D_{\mathfrak{P}_2} (M_{S_p} (K)/K)$
(resp.~$D_{\mathfrak{P}_3} (M_{S_p} (K)/K)$) 
(cf. \cite[pp.24--25]{Min}).
We also note that the image of 
\[ \{ (u_1, 1, 1) \; | \; u_1 \in \mathcal{U}_1 \} 
\to \mathcal{U}/ \mathcal{E} \]
corresponds to $I_{\mathfrak{P}_1} (M_{S_p} (K)/K)$.

We claim that
\[ I_{\mathfrak{P}_1} (M_{S_p} (K)/K) 
\cap D_{\mathfrak{P}_2} (M_{S_p} (K)/K) \cap D_{\mathfrak{P}_3} (M_{S_p} (K)/K) \]
is finite.
Since $M_{S_p} (K)/ \widetilde{K}$ is a finite extension and 
$\Gal (\widetilde{K}/ K)$ is $\mathbb{Z}_p$-torsion free, 
the assertion of (ii) follows from this claim.

In the remaining part, we shall give a proof of the above claim.
We take an element $u_1 \in \mathcal{U}_1$ such that 
the class $(u_1, 1, 1) \mathcal{E}$ corresponds to the 
element of $\mathfrak{X}_{S_p} (K)$ contained in 
$D_{\mathfrak{P}_2} (M_{S_p} (K)/K) \cap D_{\mathfrak{P}_3} (M_{S_p} (K)/K)$.
Then, there exists a non-zero integer $n$ such that 
$(u_1^n, 1, 1) \mathcal{E}$ is contained 
in both the image of $\mathcal{D}_2 \to \mathcal{U}/ \mathcal{E}$ and 
the image of $\mathcal{D}_3 \to \mathcal{U}/ \mathcal{E}$.
That is, there exist 
\[ \varepsilon, \varepsilon' \in E'_K \otimes_{\mathbb{Z}} \mathbb{Z}_p,
\; x,y \in \mathbb{Z}_p, \; u_2 \in \mathcal{U}_2, \; u_3 \in \mathcal{U}_3 \]
such that 
\[ (u_1^n \varepsilon, \varepsilon, \varepsilon) 
=(\pi_2^x, u_2, \pi_2^x) \quad \text{and} \quad
(u_1^n \varepsilon', \varepsilon', \varepsilon') 
=(\pi_3^y, \pi_3^y, u_3) \quad \text{in $\mathcal{U}$.} \]
By retaking $n$ if necessary, we may assume that 
$\varepsilon, \varepsilon'$ are the elements of 
$\langle \varepsilon_1, \varepsilon_2 \rangle \otimes_{\mathbb{Z}} \mathbb{Z}_p$.
Hence we write $\varepsilon = \varepsilon_1^{a_1} \varepsilon_2^{a_2}$ 
and $\varepsilon' = \varepsilon_1^{b_1} \varepsilon_2^{b_2}$ 
with $a_1$, $a_2$, $b_1$, $b_2 \in \mathbb{Z}_p$.
By summarizing them, we see that 
\begin{equation}\label{eq_U1}
u_1^n \varepsilon_1^{a_1} \varepsilon_2^{a_2} = \pi_2^x, \quad 
u_1^n \varepsilon_1^{b_1} \varepsilon_2^{b_2} = \pi_3^y \quad
\text{in $\mathcal{U}_1$},   
\end{equation}
\begin{equation}\label{eq_U2}
\varepsilon_1^{a_1} \varepsilon_2^{a_2} = u_2, \quad 
\varepsilon_1^{b_1} \varepsilon_2^{b_2} = \pi_3^y \quad
\text{in $\mathcal{U}_2$},   
\end{equation}
\begin{equation}\label{eq_U3}
\varepsilon_1^{a_1} \varepsilon_2^{a_2} = \pi_2^x, \quad 
\varepsilon_1^{b_1} \varepsilon_2^{b_2} = u_3 \quad
\text{in $\mathcal{U}_3$}.   
\end{equation}
The second equation of (\ref{eq_U2}) can be rewritten as 
the following:
\begin{equation}\label{eq_U12}
\varepsilon_1^{b_2} \varepsilon_2^{b_1} 
= \pi_3^y \quad
\text{in $\mathcal{U}_1$}.
\end{equation}
The first equation of (\ref{eq_U3}) can be rewritten as 
the following:
\begin{equation}\label{eq_U13}
\varepsilon_3^{a_1} \varepsilon_2^{a_2} 
(= \varepsilon_1^{-a_1} \varepsilon_2^{-a_1 + a_2})
= \pi_2^x \quad
\text{in $\mathcal{U}_1$}.
\end{equation}
Then, combining (\ref{eq_U12}), (\ref{eq_U13}) with (\ref{eq_U1}), 
we obtain the following equations:
\[ u_1^n \varepsilon_1^{a_1} \varepsilon_2^{a_2} 
= \varepsilon_1^{-a_1} \varepsilon_2^{-a_1 + a_2}, \quad
u_1^n \varepsilon_1^{b_1} \varepsilon_2^{b_2} 
= \varepsilon_1^{b_2} \varepsilon_2^{b_1} \quad \text{in $\mathcal{U}_1$}. \]
Hence 
\[ \varepsilon_1^{-2 a_1} \varepsilon_2^{-a_1} 
= \varepsilon_1^{b_2 - b_1} \varepsilon_2^{b_1 -b_2} \quad \text{in $\mathcal{U}_1$} \]
We note that the image of 
$\langle \varepsilon_1, \varepsilon_2 \rangle \otimes_{\mathbb{Z}} \mathbb{Z}_p$
in $\mathcal{U}_1$ has $\mathbb{Z}_p$-rank $2$ by Theorem A.
This implies that 
\[ -2 a_1 = b_2 -b_1, \; -a_1 = b_1 - b_2 \]
and then $a_1 =0$.
Therefore, we see that $u_1^n=1$.
Our claim follows from this.
\end{proof}

At the end of this section, we introduce the following result,
which is a specialized version of \cite[Proposition 4.B]{Min}.

\begin{propB}[Minardi \cite{Min}]
Let $\mathcal{K}/K$ be a $\mathbb{Z}_p^{\oplus d}$-extension 
($d$ is $3$ or $4$).
Assume that there exists an intermediate 
$\mathbb{Z}_p^{\oplus d-1}$-extension $\mathcal{K}'/K$ of $\mathcal{K}/K$ satisfying the 
following property:
for every $i \in \{ 1,2,3 \}$ such that $I_i (\mathcal{K}/\mathcal{K}')$ 
is not trivial, $\zprank D_i (\mathcal{K}'/K) \geq 2$.
Under this assumption, 
the pseudo-nullity of $X (\mathcal{K}')$ as a 
$\Lambda_{\Gal (\mathcal{K}'/K)}$-module implies 
the pseudo-nullity of $X (\mathcal{K})$ as a 
$\Lambda_{\Gal (\mathcal{K}/K)}$-module.
\end{propB}

\section{Proof of Theorem \ref{main_thm}}\label{Proof_main_thm}
We use the well known method, which was used in Minardi \cite{Min} 
and several other authors.
That is, we take a sequence of $\mathbb{Z}_p^{\oplus d}$-extensions 
$N^{(1)} \subset N^{(2)} \subset N^{(3)} \subset \widetilde{K}$, 
and show the pseudo-nullity inductively.

Let the notation be as in the previous sections.
In the following, we assume that all of the conditions 
(C1), (C2), (C3) of Theorem \ref{main_thm} are satisfied.

Recall that $N^*/F$ is the $\mathbb{Z}_p$-extension unramified 
outside $\mathfrak{p}^*$, and $N^{(1)} = N^* K$.
$\mathfrak{P}_1$ is unramified and finitely decomposed in $N^{(1)}$ by the assumption on (C1).
Both $\mathfrak{P}_2$ and $\mathfrak{P}_3$ are ramified in $N^{(1)}/K$. 
Moreover, both $\mathfrak{P}_2$ and $\mathfrak{P}_3$ 
are not decomposed in $N^{(1)}$ by the assumption on (C2). 
(This is satisfied even when $p=2$, because $\mathfrak{p}^*$ is decomposed in $K$.)

\begin{prop}
$X (N^{(1)})$ is finite.
\end{prop}

\begin{proof}
Note that our situation is quite similar to that of \cite[Proposition 3.5 (1)]{Kata}, 
and the following proof is also essentially the same.

In this proof, we abbreviate $M_{\mathfrak{P}_3} (K)$ to $M_3$.
By the assumption on (C3) and the fact stated in the paragraph after 
Theorem A, we see that $M_3 / K$ is a finite cyclic $p$-extension.
Hence $\mathfrak{X}_{\mathfrak{P}_3} (M_3)$ is trivial.

The unique prime of $M_3$ lying above $\mathfrak{P}_2$ 
is totally ramified in the $\mathbb{Z}_p$-extension $N^{(1)} M_3 /M_3$.
By using the argument given in the proof of 
\cite[Proposition 4.2]{Hachi}, we see that the coinvariant quotient 
$\mathfrak{X}_{\mathfrak{P}_3} (N^{(1)} M_3)_{\Gal (N^{(1)} M_3/ M_3)}$ is 
isomorphic to $\mathfrak{X}_{\mathfrak{P}_3} (M_3)$ 
(see also the proof of \cite[Proposition 3.2]{I2011}). 
Hence, $\mathfrak{X}_{\mathfrak{P}_3} (N^{(1)} M_3)_{\Gal (N^{(1)} M_3/ M_3)}$ is trivial, 
and it can be shown that $\mathfrak{X}_{\mathfrak{P}_3} (N^{(1)} M_3)$ is also trivial 
by using topological Nakayama's lemma.
Since $N^{(1)} M_3 / N^{(1)}$ is a finite extension and unramified outside $\mathfrak{P}_3$, 
we see that $\mathfrak{X}_{\mathfrak{P}_3} (N^{(1)})$ is finite.

Note that $X (N^{(1)})$ is a quotient of $\mathfrak{X}_{\mathfrak{P}_3} (N^{(1)})$. 
Hence $X (N^{(1)})$ is also finite.
\end{proof}

We shall take the $\mathbb{Z}_p^{\oplus 2}$-extension 
$N^{(2)}/K$ as the following.
Recall Lemma \ref{decomp_inertia_rank} and its proof.
Let $N^{\sharp}/K$ be the unique $\mathbb{Z}_p^{\oplus 2}$-extension 
unramified outside $\{ \mathfrak{P}_1, \mathfrak{P}_2 \}$.
Since $\zprank I_1 (N^{\sharp}/K)$ is 2, 
we see that $N^{\sharp} \cap N^{(1)}/K$ is a finite extension.
Hence $N^{\sharp} N^{(1)}/K$ is a 
$\mathbb{Z}_p^{\oplus 3}$-extension.
Since $\mathfrak{P}_2$ is ramified in $N^{(1)}/K$, we can see that 
$\zprank I_2 (N^{\sharp} N^{(1)} /N^{(1)})$ is $1$.
Then, there is a (unique) intermediate field $N^{(2)}$ of 
$N^{\sharp} N^{(1)} /N^{(1)}$ such that 
$N^{(2)}/K$ is a $\mathbb{Z}_p^{\oplus 2}$-extension and 
the prime of $N^{(1)}$ lying above $\mathfrak{P}_2$ is unramified in $N^{(2)}$.
Hence, $N^{(2)} /N^{(1)}$ is unramified outside $\mathfrak{P}_1$. 
We note that $N^{(2)} /N^{(1)}$ is ramified at every prime lying above $\mathfrak{P}_1$.

\begin{prop}
$X (N^{(2)})$ is pseudo-null as a $\Lambda_{\Gal (N^{(2)}/K)}$-module.
\end{prop}

\begin{proof}
We follow the argument given in the proof of Minardi \cite[Proposition 3.B]{Min}.
There are several similar results which are shown based on the same idea
(\cite[Proposition 4.1]{I2011}, \cite[Section 3, Step 2]{Fujii}, \cite[Proposition 3.6]{Kata}, \cite[Section 5]{Taka}).
Hence, we only give an outline for the well known part.

We denote by $\mathcal{X}$ the coinvariant quotient 
$X (N^{(2)})_{\Gal (N^{(2)}/N^{(1)})}$.
We shall show that $\mathcal{X}$ is finite.
Put $\Gamma_m = D_1 (N^{(1)}/K)$, and 
let $N^{(1)}_m$ be the fixed field of $\Gamma_m$.
Since we assumed that (C1) is satisfied, we see that $N^{(1)}_m/K$ is a finite extension. 

Let $L'$ be the intermediate field of $L (N^{(2)})/N^{(2)}$ corresponding to 
$\mathcal{X}$, then $L'$ is an abelian extension over $N^{(1)}$.
We note that $\Gal (L'/N^{(1)})$ can be considered as a $\Lambda_{\Gamma_m}$-module.
Note that $L(N^{(1)})$ is an intermediate field of $L' / N^{(1)}$.
Let $\mathcal{I}$ be the kernel of the natural surjection 
$\Gal (L' /N^{(1)}) \to X (N^{(1)})$.
We denote by $\mathcal{S}$ the set of primes of $N^{(1)}$ lying above $\mathfrak{P}_1$. 
For $P \in \mathcal{S}$, we also denote by 
$I_P$ the inertia subgroup of $\Gal (L' /N^{(1)})$ for $P$.
Note that $\Gamma_m$ acts trivially on $I_P$.
Since $\mathcal{I}$ is (topologically) generated by $I_P$ with $P \in \mathcal{S}$, 
it is a $\Lambda_{\Gamma_m}$-submodule of $\mathcal{X}$ with trivial 
$\Gamma_m$-action.
We note that $\mathcal{I}$ is finitely generated as a $\mathbb{Z}_p$-module, 
and hence $\Gal (L'/N^{(1)})$ is also.

Let $L''$ be the intermediate field of $L'/N^{(1)}$ corresponding to 
the maximal finite $\Lambda_{\Gamma_m}$-submodule of $\Gal (L'/N^{(1)})$.
Then $L''$ contains $N^{(2)}$.
By using the finiteness of $X (N^{(1)})$, we can see that 
$\Gamma_m$ acts trivially on $\Gal (L''/N^{(1)})$.
Hence $L''$ is an abelian extension over $N^{(1)}_m$.

Let $I_2$ (resp.~$I_3$) be the inertia subgroup of 
$\Gal (L''/N^{(1)}_m)$ for the unique prime lying above $\mathfrak{P}_2$ 
(resp.~$\mathfrak{P}_3$), 
and $\mathcal{I}'$ the subgroup of $\Gal (L''/N^{(1)}_m)$ (topologically) generated by 
$I_2$ and $I_3$.
Since both $I_2$ and $I_3$ have $\mathbb{Z}_p$-rank $1$, 
we see that $\zprank \mathcal{I}'$ is at most $2$.
The fixed field of $L''$ by $\mathcal{I}'$ is an 
abelian pro-$p$ extension of $N^{(1)}_m$ unramified outside $\mathfrak{P}_1$.
By Theorem A, it must be a finite extension.
Then, we conclude that $\zprank \Gal (L''/N^{(1)}_m)$ is $2$.

From the above facts, we see that $\mathcal{X} = \Gal (L'/N^{(2)})$ is finite.
By applying \cite[p.12, Lemme 4]{Per}, 
we obtain the pseudo-nullity of $X (N^{(2)})$.
\end{proof}

We recall that $N^{(2)}/N^{(1)}$ is unramified at 
every prime lying above $\mathfrak{P}_2$ or $\mathfrak{P}_3$.
We shall use Proposition B in the next step.
To apply this proposition, we need the information on the decomposition of these primes.
However, for our purpose, it is sufficient to show the following result 
(we do not determine $\zprank D_2 (N^{(2)}/K)$, $\zprank D_3 (N^{(2)}/K)$ exactly). 

\begin{lem}\label{N2_decomposition} 
The case where $\zprank D_2 (N^{(2)}/K) = \zprank D_3 (N^{(2)}/K) =1$ 
does not occur.
(Note that both $D_2 (N^{(2)}/K)$ and $D_3 (N^{(2)}/K)$ have $\mathbb{Z}_p$-rank at least $1$.)
\end{lem}

\begin{proof}
Let $N^{\mathfrak{P}_2}/K$ (resp.~$N^{\mathfrak{P}_3}/K$) 
be the unique $\mathbb{Z}_p$-extension 
such that $\mathfrak{P}_2$ (resp.~$\mathfrak{P}_3$) splits completely 
(the uniqueness follows from Proposition \ref{key_prop} (i)).

Assume that $\zprank D_2 (N^{(2)}/K) = \zprank D_3 (N^{(2)}/K) = 1$.
Then both $N^{\mathfrak{P}_2}/K$ and $N^{\mathfrak{P}_3}/K$ 
are intermediate fields of $N^{(2)}/K$. 
Since $N^{\mathfrak{P}_2} \cap N^{\mathfrak{P}_3} /K$ is a finite extension 
(see the proof of Proposition \ref{key_prop} (i)), 
we see that $N^{(2)} = N^{\mathfrak{P}_2} N^{\mathfrak{P}_3}$.

We note that $N^{\mathfrak{P}_2} N^{\mathfrak{P}_3}/K$ is contained in the fixed field of 
$D_2 (\widetilde{K}/K) \cap D_3 (\widetilde{K}/K)$.
Then, by using Lemma \ref{decomp_inertia_rank} and 
Proposition \ref{key_prop} (i), (ii), 
we see that the image of the natural mapping
\[ I_1 (\widetilde{K}/K) \to  
\Gal (\widetilde{K}/K)/ (D_2 (\widetilde{K}/K) \cap D_3 (\widetilde{K}/K)) \]
has $\mathbb{Z}_p$-rank $2$.
This implies that 
$\zprank I_1 (N^{\mathfrak{P}_2} N^{\mathfrak{P}_3}/ K) =2$.

On the other hand, $N^{(1)}$ is an intermediate field of 
$N^{(2)} = N^{\mathfrak{P}_2} N^{\mathfrak{P}_3}$ over $K$,  
and the $\mathbb{Z}_p$-extension $N^{(1)}/K$ is unramified at $\mathfrak{P}_1$.
This is a contradiction.
Then the assertion follows.
\end{proof}

We choose the $\mathbb{Z}_p^{\oplus 3}$-extension $N^{(3)}/N$ depending on the following cases. 
\begin{itemize}
\item[(a)] The prime of $N^{(1)}$ lying above $\mathfrak{P}_2$ is 
finitely decomposed in $N^{(2)}$. 
\item[(b)] The prime of $N^{(1)}$ lying above $\mathfrak{P}_2$ is 
completely decomposed in $N^{(2)}$.
\end{itemize}

For the case (a), we see that 
\[ \zprank D_1 (N^{(2)}/K) = \zprank D_2 (N^{(2)}/K) = 2. \]
By Lemma \ref{decomp_inertia_rank}, we see that $\zprank I_3 (\widetilde{K}/N^{(2)}) = 1$.
Hence we can take a $\mathbb{Z}_p^{\oplus 3}$-extension $N^{(3)}/N$
such that $N^{(3)}/ N^{(2)}$ is unramified outside 
$\{ \mathfrak{P}_1, \mathfrak{P}_2 \}$.

For the case (b), we see that the prime of $N^{(1)}$ lying above $\mathfrak{P}_3$ is 
finitely decomposed in $N^{(2)}$ by Lemma \ref{N2_decomposition}.
Hence
\[ \zprank D_1 (N^{(2)}/K) = \zprank D_3 (N^{(2)}/K) = 2. \]
Similar to the case (a), we can take 
a $\mathbb{Z}_p^{\oplus 3}$-extension $N^{(3)}/N$
such that $N^{(3)}/ N^{(2)}$ is unramified outside 
$\{ \mathfrak{P}_1, \mathfrak{P}_3 \}$.

\begin{prop}
For the $\mathbb{Z}_p^{\oplus 3}$-extension $N^{(3)}/N$ chosen above, 
$X (N^{(3)})$ is pseudo-null as a $\Lambda_{\Gal (N^{(3)}/K)}$-module.
\end{prop}

\begin{proof}
For either case (a) and (b), we can apply Proposition B for $N^{(3)}/N^{(2)}$.
\end{proof}

Now we shall finish the proof of Theorem \ref{main_thm}.

\begin{proof}[Proof of Theorem \ref{main_thm}]
By Lemma \ref{decomp_inertia_rank}, we see that all of 
$D_1 (N^{(3)}/K)$, $D_3 (N^{(3)}/K)$, 
$D_3 (N^{(3)}/K)$ have $\mathbb{Z}_p$-rank at least $2$ 
(this holds for either choice of $N^{(3)}$).
Then we can also apply Proposition B for $\widetilde{K}/N^{(3)}$.
\end{proof}

\section{Examples}\label{Examples}

We shall give examples which satisfy the conditions of Theorem \ref{main_thm}.

\begin{example}
We put $p=3$.
Let $K$ be the minimal splitting field of $X^3-10$.
We can confirm that (C1) is satisfied.
The class number of $F$ is $1$, hence (C2) is satisfied.
We can also check that $\mathfrak{X}_{\mathfrak{P}_3} (K)$ is trivial.
Thus (C3) is also satisfied.
Then, GGC for $K$ and $p$ holds by Theorem \ref{main_thm}.
\end{example}

\begin{example}
We put $p=3$.
Let $K$ be the minimal splitting field of $X^3-17$.
We can confirm that (C1) and (C2) are satisfied 
(note that the class number of $F$ is $1$).
In this case, the order of $\mathfrak{X}_{\mathfrak{P}_3} (K)$ is $3$.
However, we can also see that $\mathfrak{P}_2$ is not decomposed in 
$M_{\mathfrak{P}_3} (K)$, hence (C3) is satisfied.
Consequently, GGC for $K$ and $p$ also holds for this case by Theorem \ref{main_thm}.
\end{example}

\begin{rem}\label{rem_C1_2}
Let $M_{\mathfrak{p}^*}' (F) / F$ be the maximal abelian pro-$p$ extension 
unramified outside $\mathfrak{p}^*$ and split completely at $\mathfrak{p}$.
To confirm (C1), it is sufficient to show that 
$M_{\mathfrak{p}^*}' (F) / F$ is a finite extension.
Hachimori checked that 
$[M_{\mathfrak{p}^*}' (F) : F] = 9$ for $F = \mathbb{Q} (\sqrt[3]{10})$. 
(See \cite[Section 7.3]{Hachi}. 
He also gave partial results for other pure cubic fields.)
Kataoka also checked \cite[Assumption 3.1]{Kata} for many cases (see \cite[p.630]{Kata}).
His result includes confirming (C1) for the above examples 
(recall Remark \ref{rem_C1}). 
Note that for the examples given in the present paper, all conditions are checked separately 
by using PARI/GP \cite{PARI}.
(In this computation, the author also referred to the data and an idea of 
computation stated in Gras \cite[Appendix A]{Gras}.)
\end{rem}

\bigskip

\begin{flushleft}
Tsuyoshi Itoh \\
Division of Mathematics, 
Education Center,
Faculty of Social Systems Science, \\
Chiba Institute of Technology, \\
2--1--1 Shibazono, Narashino, Chiba, 275--0023, Japan \\
e-mail : \texttt{tsuyoshi.itoh@it-chiba.ac.jp}
\end{flushleft}

\end{document}